\documentclass{amsart}
\usepackage{latexsym,amsmath,amsthm,verbatim,ifthen,amssymb}
\usepackage[latin1]{inputenc}
\usepackage[all]{xy}
\usepackage{graphicx}
\usepackage{epsfig}
\newdir{ >}{{}*!/-7pt/\dir{>}}
\usepackage{pdfsync}

\usepackage{url}

\newtheorem{Th}{\bf Theorem}[section]
\newtheorem{Prop}[Th]{\bf Proposition}
\newtheorem{Le}[Th]{\bf Lemma}
\newtheorem{Co}[Th]{\bf Corollary}
\newtheorem{Ex}[Th]{\bf Example}
\newtheorem{Rem}[Th]{\bf Remark}

\newcommand{\NN}{\mathbb{N}}

\newcommand{\MM}{\mathbb{M}}

\newcommand{\EE}{\mathbb{E}}
\newcommand{\FF}{\mathbb{F}}

\begin{document}

\title[cofinite subsets and double negation topologies]{cofinite subsets and
double negation topologies on locales of filters and ideals}

\author{Luis Espa\~{n}ol, Jos\'e Manuel Garc\'{\i}a-Calcines, M. Carmen M\'{\i}nguez} \maketitle

\begin{abstract}
We study the role of the filter $c\mathcal{K}(X)$ of cofinite
subsets of $X$ in the locale $\mathcal{F}ilt(X)$ of all filters on
$X$, by means of the double negation topology of
$\mathcal{F}ilt(X)$, and an essential locale morphism
$\mathcal{P}(X)^{op}\to\mathcal{F}ilt(X)$. Moreover, in the case
$X=\NN$, we characterise cofinite subsets by means of the double
negation topology on the monoid $\MM$ of the maps $\NN\to\NN$ with
finite fibers, or on the submonoid $\EE\subseteq\MM$ of the
monotone and injective maps $\NN\to\NN$.
\end{abstract}


\section*{Introduction}

In order to define convergent sequences in a set $X$ endowed with
a convergence structure, cofinite subsets of $\NN$ and filters are
involved. Indeed, in general, a sequence $s:\NN\to X$ is
convergent to a point $x\in X$ if for any subset $U$ in certain
filter of subsets of $X$ containing $x$, we have that $s^{-1}(U)$
is cofinite. A basic property of convergent sequences in a broad
setting (for instance \cite{kis}, \cite{jst1}, \cite{egm}) is that
for any subsequence there exists a convergent subsequence.

In the context of exterior spaces \cite{g-g-h} (an exterior space
is nothing else than a topological space endowed with a filter of
open subsets, while an exterior map is a continuous map which
preserves the filters by inverse image) there is a notion of
convergence to filters, convergence to points being a particular
case. Namely, a sequence $s$ ``converges" to a given filter
$\mathcal{F}$ if $s^{-1}(U)$ is cofinite for any
$U\in\mathcal{F}.$ The above basic property of the convergence
stated with subsequences also works in this case. Any subsequence,
identified as a set-map $\NN\stackrel{u}{\to }\NN$, must be a
convergent sequence (in this sense) with respect to the filter
$c\mathcal{K}(\NN)$ of all cofinite subsets of $\NN$. This means
that $u$ preserves cofinite subsets for inverse image; or
equivalently, it has finite fibers. Therefore the monoid $\MM$ of
all finite fibers maps $\NN\to\NN$ plays an important role in this framework. 
This monoid and its $\MM$-sets are deeply involved in the study of the
category of exterior spaces \cite{g-h,ghr,hp}. We point out that
subsequences are usually limited to the submonoid of $\MM$
consisting of all monotone injective maps.

Any general monoid $\mbox{M}$ (we denote its operation by $f\circ
g$, recalling maps and compositions) has a locale $\Omega$ of
right ideals, which is the object of true-values for
$\mbox{M}$-sets \cite{m-m, l-l1,l-l2}. The \emph{double negation
topology} on $\mbox{M}$ is the $\mbox{M}$-subset of $\Omega$ formed by those
ideals such that $\neg\neg I=\mbox{M}$. Then, an ideal $I$ belongs
to the double negation topology if, and only if, for any $f\in
\mbox{M}$ there exists $g\in \mbox{M}$ such that $f\circ g\in I$.
This statement with quantifiers is formally similar to that used
with subsequences in the basic property of the convergence.
Therefore it seems useful the study of the relationship between
cofinite subsets and the double negation topology in monoids of
subsequences.

On the other hand, the set $\mathcal{F}ilt(X)$ of all filters on a
set $X$ is also a locale with its internal logic, negation
operator, etc. \cite{jst2,m-m}, and we consider convenient to
analyse the role of the filter $c\mathcal{K}(\NN)$ in it. In
particular its relation to the \emph{double negation topology} of
$\mathcal{F}ilt(X)$.

The first aim of this paper is to relate the filter
$c\mathcal{K}(X)$ of all cofinite subsets of $X$ to the double negation
topology of $\mathcal{F}ilt(X)$. In Section 1 (Theorem \ref{clu1})
we prove that the open sublocale of $\mathcal{F}ilt(X)$ defined by
the filter $c\mathcal{K}(X)$ is isomorphic to
$\mathcal{P}(X)^{op}$ (all subsets of $X$). In order to
characterise cofinite subsets (Theorem \ref{acont}) Section 2
concerns with the case $X=\NN$. We prove that a subset
$A\subseteq\NN$ is cofinite if, and only if, the ideal
$\mbox{Cont}(A)=\{f\in\MM;\hspace{3pt}\mbox{Im}(f)\subseteq
A\}\subseteq\MM$ belongs to the double negation topology on $\MM$.
In Section 3 we introduce the submonoid $\EE\subseteq\MM$ of all
monotone injective maps $\NN\to \NN$ and prove that the above
characterisation also works (Theorem \ref{aconte}) with the
submonoid $\EE$ instead of $\MM$. Finally, we describe (Theorems
\ref{Ldn} and \ref{Fdn}) the double negation topology on $\EE$,
and we include some notes about the bijection
$c\mathcal{K}(\NN)\cong\FF$, where $\FF\subseteq\EE$ is the
submonoid formed by the maps $u\in\EE$ such that $\mbox{Im}(u)$ is
cofinite.

\section{The locale of all filters on a set}

Let $(P,\leq)$ be any poset. By definition $P$ is a
$\wedge$-lattice if it is closed by finite meets; in particular,
it has a top element, denoted 1, and $(P,\wedge,1)$ is a monoid.
Then $x\leq y$ if and only if $x\wedge y=x$. A subset $S\subseteq
P$ is an \emph{upper subset} if $y\geq x\in S$ implies $y\in S$. A
\emph{filter} $F$ of $P$ is an upper subset which is a submonoid;
then $1\in F$, so it is nonempty. Any nonempty upper subset
contains a filter as the following result shows.

\begin{Le}\label{filinup}${}${\rm
Consider $S$ a nonempty upper subset of a given $\wedge$-lattice
$P.$ Then $S$ contains the filter $S^o=\{x\in P;\forall y\in S,
x\wedge y\in S\}.$ Moreover, $S$ is a filter if, and only if,
$S=S^o$. }
\end{Le}
\begin{proof}[Proof]
Note that the condition $x\wedge y\in S$ implies $x\in S$, so
$S^o\subseteq S.$ It is clear that $S^o$ is an upper subset since
so is $S$. Finally, $S^o$ is closed by finite meets, because given
$x,x'\in S^o$, for any $y\in S$ we have $x'\wedge y\in S$, and
$(x\wedge x')\wedge y=x\wedge(x'\wedge y)\in S,$ meaning that
$x\wedge x'\in S^o.$ The second part is clear.
\end{proof}

By $\mathcal{P}(X)$ (resp, $\mathcal{K}(X)$, $c\mathcal{K}(X)$,
$\mathcal{P}_\infty(X)$) we will denote the family of all subsets
(resp. finite subsets, cofinite subsets, infinite subsets) of a
given set $X$. $\mathcal{P}_\infty(X)$ is the complement of
$\mathcal{K}(X)$ in $\mathcal{P}(X)$, and $c\mathcal{K}(X)$ is the
family formed by the complements in $X$ of the subsets that belong
to $\mathcal{K}(X)$. If $X$ is finite, then
$\mathcal{P}(X)=\mathcal{K}(X)=c\mathcal{K}(X)$, and
$\mathcal{P}_\infty(X)=\emptyset$. We are interested in the $X$
infinite case, in particular $X=\NN$, because in this case the
cofinite subsets are nontrivial.

\medskip
A \emph{filter on} a set $X$ is a filter $\mathcal{F}$ of the
$\cap$-lattice $\mathcal{P}(X)$, that is, a (nonempty) family of
subsets of $X$ closed under finite intersections and whenever
$X\supseteq U\supseteq V\in\mathcal{F}$ we have that $U\in
\mathcal{F}$. The families $\mathcal{P}(X)$ and $c\mathcal{K}(X)$
are filters, the latter called \emph{Fr\'{e}chet filter} of $X$.
$\mathcal{K}(X)$ is a filter only when $X$ is finite. When $X$ is
infinite, $\mathcal{P}_\infty(X)$ is an upper subset but not a
filter. Taking into account that a subset $A\subseteq\NN$ is
cofinite if and only if for any infinite subset $B\subseteq\NN$,
$A\cap B$ is infinite (equivalently, there exists $B'\subseteq B$
infinite such that $B'\subseteq A$) the following result is clear:

\begin{Co}\label{cof}${}${\rm
$c\mathcal{K}(X)=\mathcal{P}_\infty(X)^o$. }\end{Co}

\medskip Any family $\{U_i;i\in I\}$ of subsets of a set $X$
generates the filter $$<U_i;i\in I>=\{U\subseteq X; \exists
i_1,...,i_n\in I, U_{i_1}\cap ... \cap U_{i_n}\subseteq U\}.$$ A
family $\{U_i;i\in I\}$ is called \emph{base filter} (or for
short, just \emph{base}) if for any $i,j\in I$ there exists $k\in
I$ such that $U_k\subseteq U_i\cap U_j$; then the generated filter
is
$$\{U\subseteq X; \exists i\in I, U_i\subseteq U\}.$$ The family
$\NN_{\geq p}=\{n\in\NN;n\geq p\}$, $p\in\NN$, is a base of the
Fr\'echet filter on $\NN$.

Important particular cases of this general construction will be
used. For any subset $A\subseteq X$, we have the filter
$\mathcal{U}(A)=\{U;A\subseteq U\}$, with
$\mathcal{U}(x)=\mathcal{U}(\{x\})$ if $x\in X$. When
$A=\emptyset$ we have the \emph{discrete filter}
$\mathcal{U}(\emptyset)=\mathcal{P}(X)$ (the greatest filter), and
when $A=X$ the \emph{indiscrete filter} $\mathcal{U}(X)=\{X\}$
(the lowest filter).

\begin{Ex}\label{int}${}${\rm
Given $A\subseteq X$, the family $\int (A)=\{U\subseteq X;A\cap
U\neq\emptyset\}$ (this notation is taken from \cite{lr}) is an
upper subset of $\mathcal{P}(X)$, which is a filter if, and only
if, $A=\{a\}$; in this case $\int (a)=\mathcal{U}(a)$. In general,
we have $\int (A)^o=\mathcal{U}(A)$. A particular case is $\int
(X)^o=\{A\subseteq X;A\neq\emptyset\}^o=\{X\}$. }
\end{Ex}

By the \emph{limit} of a filter $\mathcal{F}$ on $X,$ denoted as
$\ell(\mathcal{F})$, we mean the intersection of all the subsets of $\mathcal{F}$. The
elements of $\ell(\mathcal{F})$ are called \emph{limit points} of
$\mathcal{F}$. For instance, $\ell(\mathcal{U}(A))=A$, and
$\ell(c\mathcal{K}(X))=\emptyset$. For any subset $A\subseteq X$,
we denote by $\mathcal{C}(A)$ the the filter of all cofinite
subsets of $X$ containing $A$, that is,
$\mathcal{C}(A)=c\mathcal{K}(X)\cap\mathcal{U}(A).$ Then
$\mathcal{C}(X)=\{X\}$ and
$\mathcal{C}(\emptyset)=c\mathcal{K}(X)$. We also have
$\ell(\mathcal{C}(A))=A.$

\medskip

Let $\mathcal{F}ilt(X)$ denote the set of all filters on $X$.
Since the arbitrary intersection of filters is a filter, the set
$\mathcal{F}ilt(X)$ is a complete lattice. The supremum of a
family of filters is the filter generated by the finite
intersections of subsets in all the filters, and it is
straightforward to check that finite infima distribute arbitrary
suprema. Hence $\mathcal{F}ilt(X)$ is a frame or locale (see
\cite{jst2} or \cite{m-m}). It is important to recall that the
dual notion of a filter $\mathcal{F}\subseteq\mathcal{P}(X)$ is
that of \emph{ideal} $\mathcal{I}\subseteq\mathcal{P}(X)$, that
is, a nonempty family of subsets of $X$ such that $\mathcal{I}$ is
closed under finite unions and whenever $U\subseteq
V\in\mathcal{I}$ we have that $U\in \mathcal{I}.$ An ideal
$\mathcal{B}$ of $\mathcal{P}(X)$ is called \emph{bornology into}
$X$ with \emph{extent}
$\mbox{E}(\mathcal{B})=\bigcup\{A\in\mathcal{B}\}.$ $\mathcal{B}$
is also said to be a \emph{bornology on} $\mbox{E}(\mathcal{B})$,
and the pair $(\mbox{E}(\mathcal{B}),\mathcal{B})$ is a
\emph{bornological space} \cite{l-l2}.

Given any filter $\mathcal{F},$ the family $\mathcal{F}^b$ of all
complements $X\setminus U$ of the elements $U\in\mathcal{F}$ is a
bornology into $X$ with extent $X\setminus \ell(\mathcal{F}).$ Conversely,
given any bornology $\mathcal{B}$ on $A\subseteq X,$ one can
consider the filter $\mathcal{B}^e$ in $X$ defined by the all
complements $X\setminus U$ of elements $U\in \mathcal{B};$ then
$\ell(\mathcal{F})=X\setminus A.$ It is clear that $(-)^b$ and $(-)^e$ are
inverse monotone bijections between the sets $\mathcal{F}ilt(X)$
and $\mathcal{B}orn(X)$.

\begin{Rem}{\rm The set $\mathcal{B}orn(X)$ of all
bornologies in $X$ is the free compact regular locale generated by
$\mathcal{P}(X)$, and it is equivalent to the locale of open
subsets of the Stone-\v{C}ech compactification $\beta X$ of the
discrete topological space $X$ \cite[p. 93]{jst2}. The locale
$\mathcal{B}orn(X)$ was extensively used in \cite{l-l2}, where it
was proved that it is a subobject classifier of a Grothendieck
topos of $\NN^\NN$-sets for $\NN^\NN=\mathbf{Set}(\NN,\NN)$.}
\end{Rem}

The next results, Lemma \ref{impli1} and Theorem \ref{clu1}, are
translations to filters of those dual results on ideals
(bornologies) in \cite[p. 115]{l-l2}. However, we will include the
proof to make this article self-contained.

\begin{Le}\label{impli1}${}${\rm
The implication and the negation in the Heyting algebra
$\mathcal{F}ilt(X)$ are, respectively, the operations
$$\mathcal{F}\to\mathcal{G}=\{A\subseteq X;\forall U\in\mathcal{F},
A\cup U\in\mathcal{G}\};\hspace{15pt}\neg\mathcal{F}=\mathcal{U}(X\setminus\ell(\mathcal{F}))$$
}
\end{Le}
\begin{proof}[Proof]
Obviously, $A\in \mathcal{F}\to\mathcal{G}$ if, and only if,
$\mathcal{U}(A)\subseteq \mathcal{F}\to \mathcal{G}.$ But by the
definition of implication, this is equivalent to
$\mathcal{U}(A)\cap \mathcal{F}\subseteq\mathcal{G}$, where
$\mathcal{U}(A)\cap\mathcal{F}=\{A\cup U;U\in\mathcal{F}\}.$ For
the negation $\neg\mathcal{F}=\mathcal{F}\to\{X\},$ note that
$A\cup U=X$ means $A\supseteq X\setminus U,$ for any
$U\in\mathcal{F};$ that is
$A\in\mathcal{U}(X\setminus\ell(\mathcal{F})).$
\end{proof}

\begin{Ex} \label{exl} {\rm
Suppose $\emptyset\neq A\subsetneq X$ such that $X\setminus A$
infinite. Then the filter $\mathcal{F}=\{A\cup U; U\in
c\mathcal{K}(X\setminus A)\}$ satisfies $\ell(\mathcal{F})=A$ but
$A\notin\mathcal{F}$, with $\neg\neg\mathcal{F}=\mathcal{U}(A)$. }
\end{Ex}

The relation of the locale $\mathcal{F}ilt(X)$ with
$\mathcal{P}(X)$ is shown in the next Theorem \ref{clu1}, where it
is proved that $\mathcal{P}(X)^{op}$ is isomorphic to the open
sublocale of $\mathcal{F}ilt(X)$ defined by $c\mathcal{K}(X)$.

\begin{Th}\label{clu1}${}${\rm
$\mathcal{C},\mathcal{U}: \mathcal{P}(X)^{op}\to\mathcal{F}ilt(X)$
and $\ell:\mathcal{F}ilt(X)\to\mathcal{P}(X)^{op}$ are monotone
maps, and they satisfy the following properties:

\begin{enumerate}
\item[(i)] $\ell\circ\mathcal{C}=id=\ell\circ\mathcal{U}.$
\item[(ii)] $\mathcal{C}\dashv\ell\dashv\mathcal{U}: \mathcal{P}(X)^{op} \to\mathcal{F}ilt(X).$
\item[(iii)] $\mathcal{C}(A\cup\ell(\mathcal{F}))=\mathcal{C}(A)\cap\mathcal{F}$ (Frobenius relation).
\item[(iv)] $\mathcal{U}\circ\ell=c\mathcal{K}(X)\to(-)$ is the double negation in $\mathcal{F}ilt(X).$
\item[(v)] A filter $\mathcal{F}$ is $\neg\neg$-dense if, and only if, $\ell(\mathcal{F})=\emptyset.$
\item[(vi)] $\mathcal{F}ilt(X)_{\neg\neg}\cong\mathcal{P}(X)^{op}\cong\{\mathcal{F}\in \mathcal{F}ilt(X);\mathcal{F}\subseteq c\mathcal{K}(X)\}$.
\end{enumerate} }
\end{Th}

\begin{proof}[Proof]

(i) Obviously, $\ell\circ\mathcal{U}=id$ and
$A\subseteq\ell(\mathcal{C}(A))$ hold. Moreover, $\ell
(\mathcal{C}(A))\subseteq A;$ indeed, if
$x\in\ell(\mathcal{C}(A))$ and $x\notin A ,$ then
$X\setminus\{x\}\in\mathcal{C}(A),$ and therefore $x\in
X\setminus\{x\},$ which is a contradiction. (ii) Clearly,
$\ell(\mathcal{F})\supseteq A$ if, and only if,
$\mathcal{F}\subseteq\mathcal{U}(A)$. Now we prove that
$\mathcal{C}(A)\subseteq\mathcal{F}$ if, and only if,
$A\supseteq\ell(\mathcal{F})$: Suppose
$\mathcal{C}(A)\subseteq\mathcal{F};$ then by (i) one has
$\ell(\mathcal{F})\subseteq\ell(\mathcal{C}(A))=A.$ Conversely, if
$A\supseteq\ell(\mathcal{F})$ and $B\in\mathcal{C}(A)$ we have
$A\subseteq B$ with $X\setminus B=\{x_1,\ldots,x_n\},$ so
$\{x_1,\ldots,x_n\}\subseteq X\setminus A\subseteq X\setminus
\ell(\mathcal{F})$. Therefore for each index $i,$ $x_i\in
X\setminus U_i$ for some $U_i\in\mathcal{F}$. We conclude that
$B\in\mathcal{F}$ since $B\supseteq U_1\cap\cdots\cap
U_n\in\mathcal{F}$. (iii) By applying the adjunction
$\mathcal{C}\dashv\ell $ and part (i), it is easy to verify that
$\ell(\mathcal{F}\cap \mathcal{G})=\ell(\mathcal{F})\cup
\ell(\mathcal{G})$ and
$\mathcal{C}(A\cup\ell(\mathcal{F}))\subseteq\mathcal{C}(A)\cap\mathcal{F}$.
For the converse, obverse that $B\in\mathcal{C}(A)\cap\mathcal{F}$
means that $A\subseteq B\in\mathcal{F}$ and $B$ is cofinite. But
then $\ell(\mathcal{F})\subseteq B,$ so
$B\in\mathcal{C}(A\cup\ell(\mathcal{F})).$ (iv) By Lemma
\ref{impli1} and part (i),
$\ell(\neg\mathcal{F})=X\setminus\ell(\mathcal{F}),$ so
$\neg\neg\mathcal{F}=\mathcal{U}(X\setminus\ell
(\neg\mathcal{F}))=\mathcal{U}(\ell(\mathcal{F}))$. Moreover, by
using the adjunctions in part (ii) it results the following
equivalent relations: $A\in\mathcal{U}(\ell(\mathcal{F}))$ if, and
only if,
$\mathcal{C}(A)=\mathcal{C}(\emptyset)\cap\mathcal{U}(A)\subseteq\mathcal{F}$
if, and only if, $A\in (\mathcal{C}(\emptyset)\to\mathcal{F}).$
(v) Since $\mathcal{U}(A)=\mathcal{P}(X)$ if, and only if,
$A=\emptyset$, we have by (iv) that
$\neg\neg\mathcal{F}=\mathcal{P}(X)$ if, and only if,
$\ell(\mathcal{F})=\emptyset .$  (vi) Note that
$\neg\neg\mathcal{F}=\mathcal{F}$ if, and only if,
$\mathcal{F}=\mathcal{U}(A),$ for some $A\subseteq X;$ hence
$\mathcal{F}ilt(X)_{\neg\neg}\cong\mathcal{P}(X)^{op}$. For the
second part we observe that the adjunction $\mathcal{C}\dashv\ell$
induces an equivalence
$\mathcal{P}(X)^{op}\cong\{\mathcal{F}\in\mathcal{F}ilt(X);
\mathcal{F}=\mathcal{C}(\ell(\mathcal{F}))\}$, and
$\mathcal{F}=\mathcal{C}(\ell(\mathcal{F}))$ if, and only if,
$\mathcal{F}\subseteq c\mathcal{K}(X).$
\end{proof}

\begin{Rem}{\rm The structure of $\mathcal{F}ilt(X)$ is the simplest when $X$
is finite, since then every filter is of the form
$\mathcal{U}(X)$, $\mathcal{C}=\mathcal{U}$, and
$c\mathcal{K}(X)=\mathcal{P}(X)$.} \end{Rem}

\section{Characterizing cofinite subsets of $\NN$}

We will say that a map $f:X\to Y$ is \emph{finite} when the fiber
$f^{-1}(y)$ is finite for every $y\in Y$. It is clear that $f:X\to
Y$ is finite if, and only if $f^{-1}(A)$ is finite (resp.
cofinite) for any finite (resp. cofinite) $A\subseteq Y.$

The composition of finite maps is a finite map. The class of all
finite maps in the category $\mathbf{Set}$ of sets is the first
example of \emph{small maps} axiomatically introduced in
\cite{jm}. But now we are not interested in this direction, we
just need some easy properties of those maps. The proof of the
following lemmas can be omitted since they are straightforward.

\begin{Le}  \label{cofinite}{\rm
The following statements hold:
\begin{enumerate}
\item[(i)] If $X,Y$ are infinite, and $f:X\to Y$ is finite, then $\text{Im}(f)$ is infinite.
\item[(ii)] If the restriction of a map $f:X\to Y$ to a cofinite subset $A\subseteq X$ is injective, then $f$ is finite.
\item[(iii)] Consider maps $f,g:X\to Y$ with cofinite equalizer. Then $f$ is finite if, and only if, $g$ is finite.
\end{enumerate}}
\end{Le}

\begin{Le}  \label{cofA6}{\rm
Consider the composite $f=g\circ h:$ 
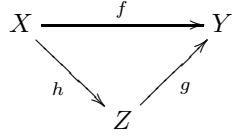
$$ \xymatrix
         {
          X\ar[rr]^f \ar[rd]_h    &  &  Y   \\
           & Z  \ar[ru]_g
         }
$$ The following statements hold:

\begin{enumerate}
\item[(i)] If $f$ is finite, then so is $h$.
\item[(ii)] If $h$ is surjective and $f$ is finite, then $g$ is also finite.
\item[(iii)] If $g$ is injective, then $f$ is finite if, and only if, $h$ is finite.
\end{enumerate}}

\end{Le}

For any set $X$, we have the monoid $\mbox{M}$ of all finite
endomaps $X\to X.$ Clearly, every injective endomap is in
$\mbox{M}$, and this monoid has not constant maps. We shall denote
$(f)=f\circ \mbox{M}=\{f\circ g;g\in \mbox{M}\}$ the principal
(right) ideal of $\mbox{M}$ generated by $f\in \mbox{M}$. Any
general monoid $\mbox{M}$ is a preordered set with the relation
defined by $f\leq g$ if there exists $h$ such that $f=g\circ h$,
that is $(f)\subseteq(g)$. A subset $I$ of $\mbox{M}$ is an ideal
if, and only if, it is an hereditary subset, that is, $f\leq g$
and $g\in I$ imply $f\in I$. When $\mbox{M}$ is a monoid of
endomaps, $f\leq g$ implies $\mbox{Im}(f)\subseteq \mbox{Im}(g)$,
but the converse is not true in general.

\begin{Le}\label{image}${}${\rm
If $f,g\in \mbox{M}$ are finite endomaps of certain set $X,$ then
$f\leq g$ if, and only if, $\mbox{Im}(f)\subseteq\mbox{Im}(g).$
 }
\end{Le}

\begin{proof}[Proof]
Given $f,g:X\to X$ such that $\mbox{Im}(f)\subseteq \mbox{Im}(g)$,
we can choose a map $h:X\to X$ such that $f=g\circ h$. Observe
that, by Lemma \ref{cofA6}(i), $h\in \mbox{M}$.
\end{proof}

\subsection{The monoid $\MM$}

We are particularly interested in the case $X=\NN$, in order to
know properties of the corresponding monoid $\MM$ of finite endomaps $\NN\to \NN$. 
We shall say that an endomap $f:\NN\to\NN$ is \emph{finally
injective} if there exists $p\in\NN$ such that the restriction
$f_{|\NN_{\geq p}}:\NN_{\geq p}\rightarrow \NN$ is injective. By
Lemma \ref{cofinite}(ii), finally injective maps are finite.

If $A\subseteq\NN$ is infinite, then the \emph{enumerating map} of
$A$ is the unique injective and order-preserving map
$u_A:\NN\to\NN$ such that $A=\text{Im}(u_A)$. It is clear that all
enumerating maps belong to $\MM$.

\begin{Le}\label{monoid}${}${\rm
The following statements about the monoid $\MM$ hold:
\begin{enumerate}
\item[(i)] Given $m,n\in\NN$ there exists $f\in\MM$ such that $f(m)=n$.
\item[(ii)] For any infinite $A\subseteq\NN$ the map $u_A$ is a split
monomorphism in $\MM$.
\item[(iii)] Given $g\in \MM$ with $A=\text{Im}(g)$, there exists a split
epimorphism $h\in\MM$ such that $g=u_A\circ h$, and there exists
an idempotent $e\in \MM$ such that $(g)=(e)$.
\end{enumerate}
 }
\end{Le}

\begin{proof}[Proof]
(i) Given $m,n\in\NN$, it is easy to check that there exists a
finally injective map $f:\NN\to \NN$ such that $f(m)=n$. (ii) If
we define the map $p_A:\NN\to\NN$ by
$p_A(n)=\min\{k\in\NN;\hspace{3pt}n\leq u_A(k)\}$, then $p_A\circ
u_A=id$. Moreover, for any $b\in\NN$, $n\in p_A^{-1}(b)$ implies
$n\leq u_A(b)$ we have that $p_A^{-1}(b)$ is finite; this proves
that $p_A$ is a finite map. (iii) By (i) the equality $(g)=(u_A)$
holds so $g=u_A\circ h$ and $u_A=g\circ k$, with $h,k\in \MM$;
subsequently, $h\circ k=id$ since $u_A$ is a monomorphism.
Obviously, the composite $e=u_A\circ p_A\in \MM$ is an idempotent
verifying $(u_A)=(e)$ by (ii).
\end{proof}

We need more notation about the set $\Omega$ of all ideals of an
arbitrary monoid $\mbox{M}$. First observe that $\Omega$ is an
$\mbox{M}$-set with the action given by the ideal $$\langle f\in I
\rangle=\{g\in \mbox{M};\hspace{3pt}f\circ g\in I\}$$ Moreover,
$\Omega$ is a locale \cite{m-m,l-l2} with the operators:
\begin{itemize}
\item[-] Implication: $I\to J=\{f\in \mbox{M};\langle f\in
I\rangle\subseteq\langle f\in J\rangle\}$,

\item[-] Negation: $\neg I=I\to\emptyset=\{f\in \mbox{M};\langle f\in
I\rangle=\emptyset\}$, and

\item[-] Double negation: $\neg \neg I=\{f\in \mbox{M};\forall g\in
\mbox{M},\exists h\in \mbox{M}, f\circ g\circ h\in I\}$.
\end{itemize}

\begin{Rem}\label{truev}${}${\rm
Clearly, any general monoid $\mbox{M}$ is an $\mbox{M}$-set, and
each $f\in \mbox{M}$ determines a unique morphism of
$\mbox{M}$-sets, denoted with the same letter: $$f:\mbox{M}\to
\mbox{M}, \hspace{10pt}f(g)=f\circ g$$ Then, for any ideal $I$ of
$\mbox{M}$, $f^{-1}((f)\cap I)=\langle f\in I \rangle$. The ideal
$(f)$ is the \emph{orbit} of $f$, and $\langle f\in I \rangle$ is
the ``measure'' in terms of ideals of the piece in the orbit of
$f$ which is contained in $I$. $\langle f\in I \rangle=\emptyset$
(respectively $\mbox{M}$) if, and only if, $(f)\cap I=\emptyset$
(respectively $(f)$). When $f$ is mono there is a bijection
$(f)\cap I\cong\langle f\in I \rangle$.}
\end{Rem}

A (\emph{Grothendieck}) \emph{topology} $\mathbf{J}$ on a monoid
$\mbox{M}$ is an $\mbox{M}$-subset of $\Omega$ such that
$\mbox{M}\in\mathbf{J}$ and verifies the following \emph{local
condition}:
\begin{enumerate}
\item[(LC)] Consider ideals $I\in\Omega$, $J\in\mathbf{J}$.
If $\langle f\in I \rangle\in\mathbf{J}$ for any $f\in J,$ then $I\in\mathbf{J}$.
\end{enumerate}

The set of all topologies in $\mbox{M}$ forms a complete lattice,
with bottom element $\{\mbox{M}\}$ and top element $\Omega$. Any
monoid $\mbox{M}$ has the \emph{double negation topology},
constituted by the ideals $I$ satisfying $\neg\neg I=\mbox{M}$
($\neg\neg-$\emph{dense} ideals), that is,
$$\mathbf{J}_{\neg\neg}=\{I\in\Omega;\forall f\in \mbox{M},
\exists g\in \mbox{M}, f\circ g\in I\}.$$

\begin{Rem}\label{dn}${}${\rm
The condition $\neg\neg I=\mbox{M}$ is equivalent to $\neg
I=\emptyset$. Moreover, if $I\in\mathbf{J}_{\neg\neg}$ then
$\mathbf{J}_{\neg\neg}=\{J\in\Omega;I\subseteq \neg\neg J\}$. This
description is useful when there is a special ideal in
$\mathbf{J}_{\neg\neg}$.
 }
\end{Rem}

It is well known that every topology is a filter in $\Omega$,
though the converse is not true. The filter
$\{J\in\Omega;I\subseteq J\}$ is a topology if $I$ is a two-sided
and idempotent ideal \cite{k-l,l-l1}. By analogy with Example
\ref{int}, given an ideal $I\in\Omega$ we can consider the family
of ideals $\int(I)=\{J\in\Omega;I\cap J\neq\emptyset\}$, which is
an upper subset of $\Omega$. We have a monotone map
$\int:\Omega\to\mathcal{P}(\Omega)$ with
$\int(\emptyset)=\emptyset$ and
$\int(\mbox{M})=\{I\in\Omega;I\neq\emptyset\}$.

\begin{Rem}\label{at}${}${\rm
We observe that $\int(\mbox{M})$ is a topology if, and only if,
$\mbox{M}$ is atomic (i.e., for any $f,g\in \mbox{M}$, there
exists $h,k\in \mbox{M}$ such that $f\circ h=g\circ k$). In this
case the topology is called the \emph{atomic topology}, denoted
$\mathbf{J}_{at}$, and it satisfies
$\mathbf{J}_{at}=\mathbf{J}_{\neg\neg}$ (particular case of atomic
site defined in \cite{b-d}).
 }
\end{Rem}

\begin{Prop}\label{odn}${}${\rm
$\int (\mbox{M})^o=\mathbf{J}_{\neg\neg}$. }
\end{Prop}

\begin{proof}[Proof]
Given $I\in\Omega$, we have that $I\in\int (\mbox{M})^o$ if, and
only if, $I\cap J\neq\emptyset$ for any ideal $J\neq\emptyset .$
In other words, $I\cap J=\emptyset$ implies $J=\emptyset$; but
$I\cap J=\emptyset$ means $J\subseteq\neg I$, so $\neg
I=\emptyset$ (equivalently $\neg\neg I=\mbox{M}$). Conversely, if
$I\in\mathbf{J}_{\neg\neg}$ then $\neg I=\emptyset$ and therefore
$I\in \int (\mbox{M})^o$.
\end{proof}

The above description of $\mathbf{J}_{\neg\neg}$ is general for
any monoid, but we can characterise the double negation topology
on $\MM$ in the particular language of subsets.

\begin{Prop}\label{dnM}${}${\rm
Suppose an ideal $I$ of $\MM$. Then:
\begin{enumerate}
\item[(i)] $\neg I=\{f\in \MM;\hspace{3pt}\forall g\in
I,\mbox{Im}(f)\cap\mbox{Im}(g)\;\text{finite}\}$.

\item[(ii)] $I\in\mathbf{J}_{\neg \neg}$ if, and only if, $\forall
f\in\MM$, $\exists g\in I$, $\mbox{Im}(f)\cap \mbox{Im}(g)$ is
infinite. \end{enumerate} }
\end{Prop}

\begin{proof}[Proof]
(i) Observe that $f\in\neg I$ means $(f)\cap I=\emptyset$. Now,
take $g\in I$ and suppose that $A=\mbox{Im}(f)\cap \mbox{Im}(g)$
is infinite. Then, by Lemma \ref{image}, $u_A$ factorises through
both maps $u_A=f\circ f'=g\circ g'$, with $f',g'\in\MM$. Therefore
$u_A\in(f)\cap I$ and $A$ must be finite. The converse is clear.
(ii) Use the fact that $I\in\mathbf{J}_{\neg \neg}$ if, and only
if, $\neg I=\emptyset$.
\end{proof}

For any subset $A\subseteq\NN$, the set
$\mbox{Cont}(A)=\{f\in\MM;\hspace{3pt}\mbox{Im}(f)\subseteq A\}$
is an ideal of $\MM$, called the \emph{content} of $A$.

\begin{Th}\label{acont}{\rm
For any subset $A\subseteq\NN$:
\begin{enumerate}
\item[(i)] $A$ is finite if, and only if, $\mbox{Cont}(A)=\emptyset$.
\item[(ii)] $A$ is cofinite if, and only if, $\mbox{Cont}(A)\in\mathbf{J}_{\neg \neg }$.
\end{enumerate}}
\end{Th}

\begin{proof}[Proof]
(i) If $A$ is finite, then $\mbox{Cont}(A)=\emptyset$ since any
$f\in\MM$ has infinite image by Lemma \ref{cofinite}(i). If $A$ is
infinite then $u_A\in \mbox{Cont}(A)$. (ii) We will use the fact
that, by Corollary \ref{cof}, a subset $A$ is cofinite if, and
only if, for any infinite subset $B$, $A\cap B$ is infinite.
Indeed, suppose $A$ cofinite and take $f\in\MM.$ Then
$B=\mbox{Im}(f)$ is infinite, so $A\cap B$ and $C=f^{-1}(A\cap B)$
are infinite. Considering $g=u_C$ in $\MM$, we have that $f\circ
g\in\MM$ with $\mbox{Im}(f\circ g)\subseteq A,$ that is, $f\circ
g\in \mbox{Cont}(A).$ This means that
$\mbox{Cont}(A)\in\mathbf{J}_{\neg \neg }$. Conversely, given $B$
infinite, we take $f=u_B$. Since
$\mbox{Cont}(A)\in\mathbf{J}_{\neg \neg }$ there exists $g\in\MM$
such that $f\circ g\in \mbox{Cont}(A)$. Taking the infinite subset
$C=\mbox{Im}(g)$ we have that $f(C)\subseteq A\cap B$. But $f(C)$
is infinite because $f$ is injective. We conclude that $A\cap B$
is infinite.
\end{proof}

\begin{Rem}{\rm After Proposition \ref{odn} and Theorem \ref{acont},
$A\subseteq\NN$ is infinite if and only if $\mbox{Cont}(A)\in\int
(\mbox{M})$, and $A\subseteq\NN$ is cofinite if and only if
$\mbox{Cont}(A)\in\int (\mbox{M})^o$.} \end{Rem}

In the proof of Theorem \ref{acont}(i) we saw that $u_A\in
\mbox{Cont}(A)$ when $A\subseteq\NN$ is infinite. This relation
can be improved as follows:

\begin{Prop}\label{ainf}{\rm
Consider a subset $A\subseteq\NN$ and its associated ideal
$\mbox{Cont}(A)$. Then $A$ is infinite if, and only if,
$\mbox{Cont}(A)$ is principal. }
\end{Prop}

\begin{proof}[Proof]
If $A$ is infinite then, obviously, $(u_A)\subseteq
\mbox{Cont}(A)$. Moreover, given any $u\in \mbox{Cont}(A)$, we
have $u=u_B$ with $B\subseteq A$; that is, $u\leq u_A$. But this
fact implies that $u\in(u_A)$; hence $(u_A)=\mbox{Cont}(A)$.
Conversely, if $\mbox{Cont}(A)=(f)$ for some $f\in \MM,$ then
$\mbox{Im}(f)\subseteq A$, and therefore $A$ is infinite since so
is $\mbox{Im}(f)$.
\end{proof}

Now we consider the submonoid $\EE\subseteq\MM$ formed by all the
enumerating maps $u_A$ with $A\subseteq\NN$ infinite, so that
$\EE\cong\mathcal{P}_\infty(X)$. $\EE$ is constituted by the maps
$u:\NN\to\NN$ such that $m<n$ implies $u(m)<u(n)$; hence $u(n)\leq
n$ for any $n\in\NN$, and the fixed points of $u$ gives an initial
subset of $\NN$. As $\MM$, the submonoid $\EE$ is not abelian and
has not constants.

In the proof of Theorem \ref{acont}(ii) the following two general
conditions were considered for the ideal $\mbox{Cont}(A)$:
\begin{enumerate}
\item[(a)] $\forall f\in\MM, \exists v\in\EE, f\circ v\in I$.

\item[(b)] $\forall f\in\MM, \exists g\in\MM, f\circ g\in I$.
\end{enumerate}

It was proved (a), that clearly implies (b), and the converse is
also true since we have the following result:

\begin{Prop}\label{E}${}${\rm
For every $f\in \MM$, there exists $u\in\EE$ such that $f\circ u\in\EE$.
}
\end{Prop}

\begin{proof}[Proof]
Taking $B=\{n\in\NN;k<n\Rightarrow f(k)<f(n)\}$ we have that
$f\circ u_B\in\EE$. Observe that $B$ is infinite because $f\in
\MM$.
\end{proof}

Therefore, we are led to consider whether the monoid $\MM$ can be
replaced by the monoid $\EE$ in Theorem \ref{acont}. For any ideal
$I$ of $\MM$, $I_\EE=I\cap\EE$ is an ideal in $\EE$. In
particular, when $A\subseteq\NN$ is infinite we have that
$u_A\in\EE$ and we can consider the principal ideal
$(u_A)=u_A\circ\MM$ in $\MM$, and the ideal
$(u_A)_\EE=(u_A)\cap\EE$. Note that $(u_A)_\EE$ is precisely the
principal ideal generated by $u_A$ in $\EE$, that is,
$(u_A)_\EE=u_A\circ\EE$. Hence Proposition \ref{ainf} can be also
stated for the monoid $\EE$:

\begin{Prop}\label{ainfe}{\rm
If $A\subseteq\NN$ is infinite, then
$\mbox{Cont}_\EE(A)=(u_A)_\EE$. }
\end{Prop}

\begin{Prop}\label{idEM}{\rm
For any ideal $I$ of $\MM$, $I\in\mathbf{J}_{\neg\neg}$ if, and
only if $I_\EE\in\mathbf{J}_{\neg\neg}(\EE)$. }
\end{Prop}

\begin{proof}[Proof]
The equivalent conditions (a) and (b) above imply

\medskip
(c) $\forall u\in\EE, \exists v\in\EE, u\circ v\in I$.

\medskip
Moreover in this condition (c) we can replace the ideal $I$ of
$\MM$ by the ideal $I_\EE$ of $\EE$; this means that
$I_\EE\in\mathbf{J}_{\neg\neg}(\EE)$. Now, if (c) holds true, then
for any $f\in\MM$ there exists  $w\in\EE$ such that $f\circ
w\in\EE$ (see Proposition \ref{E}). But then there exists
$v\in\EE$ such that $(f\circ w)\circ v\in I$, that is, there
exists $u=w\circ v\in\EE$ such that $f\circ u\in I_\EE\subseteq
I$.
\end{proof}

As a corollary of Proposition \ref{idEM}, Theorem \ref{acont} has a version for $\EE$:

\begin{Co}\label{aconte}{\rm
For any $A\subseteq\NN$:
\begin{enumerate}
\item[(i)] $A$ is finite if, and only if, $\mbox{Cont}_\EE(A)=\emptyset$.
\item[(ii)] $A$ is cofinite if, and only if, $\mbox{Cont}_\EE(A)\in\mathbf{J}_{\neg \neg}(\EE)$.
\end{enumerate}}
\end{Co}

\begin{proof}[Proof]
(i) Consider Theorem \ref{acont} and Proposition \ref{ainfe}. (ii)
See Proposition \ref{idEM}.
\end{proof}

\section{The monoid $\EE$}

From now on, we will only consider the monoid $\EE$ so there will
be not need to make it explicit in the notation. Since each
element of $\EE$ is injective, the unique idempotent or invertible
element in $\EE$ is the identity, and for any ideal $I$ the
bijection $\langle u\in I\rangle\cong (u)\cap I$ holds 
(see Remark \ref{truev}). Moreover, this monoid is a poset, with 
$u_A\leq u_B$ given by $A\subseteq B$. We have the identity as top 
element, but meets do not exist, since in general the intersection 
of infinite subsets is not an infinite subset.

Involving the order relation, we have $I\in\mathbf{J}_{\neg\neg}$
if, and only if, for any $u\in\EE$, there exists $v\leq u$, such
that $v\in I$. In other words, for any $A\subseteq\NN$ infinite,
there exists $B\subseteq A$ infinite such that $u_B\in I$.

We will call \emph{extent} of an ideal $I$ the infinite set
$$\mbox{Ext}(I)=\bigcup_{u\in I}\mbox{Im}(u)$$ so $I\subseteq(u_{Ext(I)})$. An
ideal $I$ of $\EE$ is called \emph{extended} if $\mbox{Ext}(I)$ is
cofinite.

\begin{Le} \label{extcof}
If $I\in\mathbf{J}_{\neg\neg}$, then $I$ is extended.
\end{Le}

\begin{proof}[Proof] Let $A$ be the complement of $\mbox{Ext}(I)$.
If $A$ is infinite, then $u=u_A\in\EE$, and $\mbox{Im}(u\circ
v)\subseteq A$ for any $v\in\EE$, so $\langle u\in
I\rangle=\emptyset$. Hence $I\notin\mathbf{J}_{\neg\neg}$.
\end{proof}

\begin{Rem}{\rm
The reciprocal of Lemma \ref{extcof} is false. Indeed, given
$a\in\NN$, we can consider the map $u_a(n)=an$. Then the ideal
$I=\bigcup_{1<a}(u_a)$ satisfies
$\mbox{Ext}(I)=\NN\setminus\{1\}$, but it does not belong to
$\mathbf{J}_{\neg\neg}$ since, for instance, the map $p$
enumerating the set of prime numbers does not satisfy $u\leq p$
for every $u\in I$.}
\end{Rem}

Corresponding to cofinite subsets of $\NN$, we have the submonoid
$\FF\subseteq\EE$ defined by $u_A\in\FF$ if, and only if, $A$ is
cofinite. Hence the bijection $\EE\cong\mathcal{P}_\infty(X)$
induces the bijection $\FF\cong c\mathcal{K}(X)$. If $A$ is
cofinite and $K$ is its finite complement, the we also denote
$u_A=\sigma_K$; in particular, we write $\sigma_n=\sigma_{\{n\}}$,
$n\in\NN$, and the successor map $\sigma\in\mathbb{F}$,
$\sigma(n)=n+1$, is $\sigma=\sigma_0$. Corollary \ref{cof} has a
correlative result in the level of monoids $\FF\subseteq\EE$.

\begin{Prop}\label{cofmon}${}${\rm
$\FF=\{u\in\EE;\hspace{3pt}\forall v\in\EE, \exists u\wedge
v\in\EE \}$. }
\end{Prop}

\begin{proof}[Proof]
If $u_A\in\FF$ and $u_B\in\EE$, then $A\cap B$ is infinite. It is
easy to verify that $u_A\wedge u_B=u_{A\cap B}$ is a meet in
$\EE$. Conversely, if $A$ and $B=X\setminus A$ are infinite, then
the meet $u_A\wedge u_B$ does not exist.
\end{proof}

We will consider $L=\EE\setminus\FF$. If $u_A\in L$, then $A$ is
an infinite subset of $\NN$ with an infinite complementary subset.
It is clear that $L$ is an ideal of $\EE$, so the monoid is
partitioned in an ideal and a submonoid. Therefore if $u\circ v\in
L$ and $u\in\mathbb{F}$ (resp. $v\in\mathbb{F}$), then $v\in L$
(resp. $u\in L$). The ideal $L$ has this interesting property: $L$
is in bijection with the set of all Galois connections $f\dashv
g:\NN\to\NN$ of the poset $\NN$ \cite{lam}. Our next aim is to
describe $\mathbf{J}_{\neg\neg}$ on $\mathbb{E}$ by using $L$ and
$\mathbb{F}$.

\begin{Th}\label{Ldn}${}${\rm
The following statements hold:
\begin{enumerate}
\item[(i)] $L$ is an idempotent two-sided ideal of $\EE$.

\item[(ii)] $\langle u\in L\rangle=L$ for every $u\in\mathbb{F}$.

\item[(iii)] $L\in\mathbf{J}_{\neg\neg}$.
\end{enumerate}
}
\end{Th}

\begin{proof}[Proof]
(i) The fact of being two-sided is clear as well as that $L\circ
L\subseteq L$. We also have that $L\subseteq L\circ L$; indeed,
given $u=u_A\in L$, both $A$ and its complement $A'$ are infinite
and we can take two infinite sets $B,C$ such that $B\cup C=A'$ and
$B\cap C=\emptyset$. This way we obtain $v=u_{A\cup B}\in L$. Now
if we take $D=v^{-1}(A)$, then the function $w=u_D$ belongs to $L$
and $v\circ w=u_{v(D)}=u$. (ii) As $L$ is two-sided,
$L\subseteq\langle u\in L\rangle$ for any $u\in\EE$. For the
converse suppose $u\in\FF$. If $u\circ v\in L$, then $v\in L$
because $v\in\mathbb{F}$ implies $u\circ v\in\mathbb{F}$, which is
false. (iii) After Remark \ref{dn} we must prove that $\neg
L=\emptyset$. But this is true since, by (ii), $\langle u\in
L\rangle=L\; \text{or}\; \EE$.
\end{proof}

By Remark \ref{dn} and Theorem \ref{Ldn}(iii) we have
$$\mathbf{J}_{\neg\neg}=\{I\in\Omega;L\subseteq \neg\neg I\}$$ The
relation $L\subseteq\neg\neg I$ is always strict because
$L=\neg\neg I$ implies $L=\neg\neg L=\EE$, which is false. By
Theorem \ref{Ldn}(i), the family $[L)$ of all ideals greater or
equal than $L$ is a topology on $\EE$ \cite[Theorem 4.4]{k-l}, and
$[L)\subset\mathbf{J}_{\neg\neg}$. However, the converse relation
is false: if we consider the successor map, then we have
$(\sigma)=\{u\in\mathbb{E};\hspace{3pt}u(0)\not=0\}$ and
$\neg\neg(\sigma)=\mathbb{E}$, so $L\subseteq\neg\neg(\sigma)$ is
true but $L\subseteq(\sigma)$ is false.

\begin{Ex}\label{exdn}${}${\rm
We give some examples of ideals of $\EE$ belonging or not to
$\mathbf{J}_{\neg\neg}$:
\begin{enumerate}
\item[(i)] Since $L\in\mathbf{J}_{\neg\neg}$,
$\langle u\in L\rangle\in\mathbf{J}_{\neg\neg}$, for any
$u\in\EE$.

\item[(ii)] $(u)\in \mathbf{J}_{\neg\neg}$ if, and only if, $u\in\FF$. Hence
$\int(\FF)\subseteq\mathbf{J}_{\neg\neg}$.

\item[(iii)] If $u_A,u_B\in\EE$ and $A\cup B$ is cofinite, then
$(u_A)\cup(u_B)\in\mathbf{J}_{\neg\neg}$.

\item[(iv)] If $u\in\FF$, then $u\circ L\in\mathbf{J}_{\neg\neg}$
and $L\cup(u)$ is an ideal of $\mathbf{J}_{\neg\neg}$ which is
strictly between $L$ and $\EE$. \end{enumerate} }
\end{Ex}

\begin{Th}\label{Fdn}${}${\rm
An ideal $I$ belongs to $\mathbf{J}_{\neg\neg}$ if, and only if:
\begin{enumerate}
\item[(i)] $I$ is an extended ideal, and therefore
$u=u_{Ext(I)}\in\FF$ with $I\subseteq(u)$; and

\item[(ii)] For any $v\in L$ satisfying $v\leq u$ there exists $w\leq
v$ such that $w\in I$. \end{enumerate} }
\end{Th}

\begin{proof}[Proof]
If $I\in\mathbf{J}_{\neg\neg}$, then we have (i) by Lemma
\ref{extcof}; also (ii) is clear. Conversely, suppose (i) and (ii)
hold true. Then we have to check that for any $v\in\EE$ there
exists $w\leq v$ such that $w\in I$. Indeed, if $v\in\FF$, then,
by Proposition \ref{cofmon} we obtain $w=u\wedge v\in I$, where
$u\in \FF$ is constructed as in (i). In the case $v\in L$ we have
that $u\wedge v\in L$ and therefore we can apply (ii) to get
$w\leq u\wedge v\leq v$ such that $w\in I$.
\end{proof}

\subsection{The monoid $\FF$}

Finally we recall some properties of the monoid $\FF$ with the
order preserving bijection $\FF\cong c\mathcal{K}(X)$. If
$K=\{k_1,\dots,k_r\}\subseteq\NN$\, with\, $k_1<\dots<k_r\,$, then
$\sigma_K=\sigma_{k_r}\circ\cdots\circ\sigma_{k_1}$.
$F=\FF\setminus\{id\}$ is a two-sided but not idempotent ideal.
$\FF$ is the free $\vee-$semilattice generated by the poset $\NN$,
and the set $\Omega$ of all ideals of $\FF$ is the free complete
Heyting algebra generated by $\mathbb{N}$, with universal map the
sequence $\mathbb{N}\rightarrow\Omega$ of principal ideals
$(\sigma_n)$ \cite[p. 22]{j-t}. $\FF$ is an atomic monoid, hence
the intersection on nonempty ideals in $\FF$ is a nonempty ideal
of $\FF$, and (Remark \ref{at}) its double negation topology of
$\FF$ is the atomic topology
\begin{center}$\int(\FF)=\int(\FF)^o=\{I\in\Omega;I\neq\emptyset\}$\end{center}

\end{document}